\documentclass[reqno,11pt]{amsart}
\usepackage{amsmath,amsthm,amscd,amssymb,graphicx,enumerate,latexsym}

\usepackage[raiselinks,colorlinks]{hyperref}
\hypersetup{citecolor=blue}

\numberwithin{equation}{section}

\theoremstyle{plain}
\newtheorem{thm}{Theorem}[section]
\newtheorem{lemma}[thm]{Lemma}

\newtheorem{conj}[thm]{Conjecture}

\theoremstyle{definition}

\theoremstyle{remark}
\newtheorem{remark}{Remark}[section]

\def\Re{\mathop{\rm Re}\nolimits}

\newcommand{\s}{\text{\rm{s}}}

\DeclareMathOperator*{\supp}{supp}
\DeclareMathOperator*{\arccosh}{arccosh}
\allowdisplaybreaks

\title[Szeg\H o theorems with a critical point of arbitrary order]{On higher-order Szeg\H o theorems with a single critical point of arbitrary order}
\author{Milivoje Lukic}
\address{Rice University, 6100 Main Street, Mathematics MS 136, Houston, TX 77005}
\date\today
\email{milivoje.lukic@rice.edu}
\thanks{The author was partially supported by NSF Grant DMS-1301582}


\begin{document}

\begin{abstract}
We prove the following higher-order Szeg\H o theorems: if a measure on the unit circle has absolutely continuous part $w(\theta)$ and Verblunsky coefficients $\alpha$ with square-summable variation, then for any positive integer $m$,
\[
\int (1-\cos \theta)^m \log w(\theta) d\theta
\]
is finite if and only if $\alpha \in \ell^{2m+2}$.

This is the first known equivalence result of this kind in the regime of very slow decay, i.e.\ with $\ell^p$ conditions with arbitrarily large $p$. The usual difficulty of controlling higher-order sum rules is avoided by a new test sequence approach.
\end{abstract}

\maketitle
\section{Introduction}

A Borel probability measure $\mu$ on the unit circle $\partial\mathbb{D}$ whose support isn't a finite set corresponds bijectively to its sequence of Verblunsky coefficients $\alpha = \{\alpha_n \}_{n=0}^\infty \in \mathbb{D}^\infty$ \cite{Verblunsky35}. The correspondence is given in terms of orthogonal polynomials $\varphi_n(z)$, $\deg \varphi_n = n$, which are obtained by orthonormalizing the sequence $1, z, z^2, \dots$ with respect to $\mu$ and obey a recursion relation in terms of $\alpha_n$,
\[
\varphi_{n+1}(z) = \frac 1{\sqrt{1-\lvert \alpha_n\rvert^2}} \left( z \varphi_n(z) - \bar \alpha_n z^n \overline{ \varphi_n(1 / \bar z) } \right).
\]
Alternatively, one can construct from $\alpha$ a five-diagonal matrix on $\ell^2(\mathbb{N}_0)$, called a CMV matrix after \cite{CanteroMoralVelazquez03}, which is unitary and has $\mu$ as its spectral measure. That construction shows that the correspondence between $\alpha$ and $\mu$  is similar to the correspondence between a Jacobi matrix  or Schr\"odinger operator and its spectral measure, and puts this correspondence in the realm of spectral theory.

Let us denote the Lebesgue decomposition of $\mu$ by
\[
d\mu = w(\theta) \frac{d\theta}{2\pi} + d\mu_s.
\]
Szeg\H o's theorem \cite{Szego20,Szego21,Verblunsky36} states that $\alpha \in \ell^2$ is equivalent to
\[
\int \log w(\theta) \frac{d\theta}{2\pi} > -\infty.
\]
Such equivalence results, between an operator condition and a spectral measure condition, are rare in spectral theory and are of great interest. In the last fifteen years, Szeg\H o's theorem has been carried over  to Schr\"odinger operators by Deift--Killip~\cite{DeiftKillip99} and Killip--Simon~\cite{KillipSimon09} and to Jacobi matrices by Killip--Simon~\cite{KillipSimon03}, and higher-order Szeg\H o theorems have been the subject of many papers \cite{MolchanovNovitskiiVainberg01,LaptevNabokoSafronov03,Kupin04,Kupin05,OPUC1,NazarovPeherstorferVolbergYuditskii05,SimonZlatos05,GolinskiiZlatos07}, but the general conjecture for them remains unsolved.

Let $\delta$ denote the forward derivative on sequences,
\[
(\delta\alpha)_n = \alpha_{n+1} - \alpha_n.
\]
Our main result is the following higher-order Szeg\H o theorem.

\begin{thm}\label{T1.1}
Assume that $\delta \alpha\in \ell^2$. For any $m\in \mathbb{N}_0$, the condition
\begin{equation}\label{1.1}
\int (1-\cos \theta)^m \log w(\theta) \frac{d\theta}{2\pi} > -\infty
\end{equation}
is equivalent to $\alpha \in \ell^{2m+2}$.
\end{thm}

\begin{remark}\label{R1.1}
The integral \eqref{1.1} is always well defined and never $+\infty$ since the positive part of the integrand is $L^1$. This is because $ (1-\cos\theta)^m \le 2^m$ and $\log w(\theta) \le w(\theta)$ imply
\begin{equation}\label{1.2}
\int \bigl( (1-\cos\theta)^m \log w(\theta)\bigr)_+ \frac{d\theta}{2\pi} \le \int  2^m w(\theta) \frac{d\theta}{2\pi} \le 2^m.
\end{equation}
\end{remark}

Theorem~\ref{T1.1} is similar to results of Molchanov--Novitskii--Vainberg~\cite{MolchanovNovitskiiVainberg01} and Kupin~\cite{Kupin05} for Schr\"odinger and Jacobi operators, but those results are only implications in one direction, from an operator condition to a conclusion about the measure; moreover, \cite{Kupin05} uses additional technical assumptions and \cite{MolchanovNovitskiiVainberg01} concludes presence of a.c.\ spectrum but not an integral condition on it.

Theorem~\ref{T1.1} is the first known equivalence statement of this kind with $\ell^p$ conditions with arbitrarily large $p$, and we expect that the ideas used here can be extended to further improve our understanding of higher-order Szeg\H o theorems, not only for CMV matrices but also for Jacobi and Schr\"odinger operators.

Theorem~\ref{T1.1} is a special case of the following conjecture of Simon~\cite{OPUC1}.

\begin{conj}[{\cite[Section 2.8]{OPUC1}}] \label{C1.3}
For any $m\in \mathbb{N}_0$, \eqref{1.1} is equivalent to $\alpha \in \ell^{2m+2}$ and $\delta^m \alpha \in \ell^2$.
\end{conj}

Besides Szeg\H o's theorem ($m=0$), Conjecture~\ref{C1.3} has been proved for $m=1$ by Simon~\cite{OPUC1} and for $m=2$ by Simon--Zlato\v s~\cite{SimonZlatos05}. Golinskii--Zlato\v s~\cite{GolinskiiZlatos07} have proved a partial equivalence for any $m\in \mathbb{N}_0$: if $\alpha \in \ell^4$, then \eqref{1.1} is equivalent to $\delta^m \alpha \in \ell^2$. However, none of these previous results go beyond $\ell^6$ conditions, since they rely on manual manipulations of an expression involving $\alpha$ (see \eqref{2.1}, \eqref{2.2} below) and this expression very quickly becomes prohibitively complicated. 

Our result can also be motivated from a different point of view, by a result from Lukic \cite{Lukic8}, which is an extension of Jacobi matrix results of Denisov~\cite{Denisov09} and Kaluzhny--Shamis~\cite{KaluzhnyShamis12}.

\begin{thm}[\cite{Lukic8}] \label{T1.2} If $\delta \alpha \in \ell^2$ and $\lim_{n\to\infty} \alpha_n = 0$, then for every closed arc $I \subset\partial \mathbb{D} \setminus \{1\}$,
\[
\int_I \log w(\theta) \frac{d\theta}{2\pi} > -\infty.
\]
\end{thm}

From this point of view, integrability of $\log w$ away from $\theta=0$ is given and Theorem~\ref{T1.1} relates behavior around the critical point $\theta=0$ to decay properties of $\alpha$.

We rely on the approach of \cite{OPUC1,SimonZlatos05,GolinskiiZlatos07} and in particular a formula of \cite{GolinskiiZlatos07} that expresses the integral \eqref{1.1} in terms of the Verblunsky coefficients. We avoid the limitations of this approach by only using it to prove the following.

\begin{thm}\label{T1.4}
Assume that $\delta \alpha\in \ell^2$. For any $m\in \mathbb{N}_0$, there is an $l\in \mathbb{N}$ with $l\le m+1$ such that the condition \eqref{1.1} is equivalent to $\alpha \in \ell^{2l}$.
\end{thm}

It then remains to show that $l=m+1$, for which we use a new test sequence approach. We rely on the following result about the asymptotic behavior of the a.c.\ part of the measure near the critical point.

\begin{thm}\label{T1.5}
Assume that $\delta\alpha \in \ell^1$. If $\alpha \in \ell^p$ for some $p\in\mathbb{N}$, then $\supp \mu_\s \subset\{1\}$, $w(\theta)$ is strictly positive and continuous on $e^{i\theta} \in \partial\mathbb{D} \setminus \{1\}$, and
\[
\lvert \log w(\theta) \rvert = O \left( \frac 1{\lvert \theta \rvert^{p-1}}\right), \quad \theta \to 0.
\]
\end{thm}

Given these two results, the proof of Theorem~\ref{T1.1} is almost immediate.

\begin{proof}[Proof of Theorem~\ref{T1.1}]
Pick $\alpha$ to be a sequence of Verblunsky coefficients with $\delta \alpha \in \ell^1$ and $\alpha \in \ell^{2m+1} \setminus \ell^{2m}$; for instance, take
\begin{equation}\label{1.3}
\alpha_n = (n+2)^{- 1/ (2m)}.
\end{equation}
By Theorem~\ref{T1.5}, $\lvert \log w(\theta) \rvert = O( \lvert \theta \rvert^{-2m})$, so \eqref{1.1} holds.

By Theorem~\ref{T1.4}, \eqref{1.1} is equivalent to $\alpha \in \ell^{2l}$ for some $l\in \mathbb{N}$ with $l \le m+1$. The measure corresponding to Verblunsky coefficients \eqref{1.3} obeys \eqref{1.1} but $\alpha \notin \ell^{2m}$, so $l> m$. Thus, $l=m+1$.
\end{proof}

We think of Theorem~\ref{T1.1} as a first application of this test sequence approach and hope that this approach can be used to make further progress on higher-order Szeg\H o theorems.

Theorem~\ref{T1.5} is related to Weidmann's theorem \cite{Weidmann67,PeherstorferSteinbauer00}, which states that $\delta\alpha \in \ell^1$ and $\lim_{n\to\infty} \alpha_n = 0$ imply that $\supp \mu_s \subset \{1\}$ and $w(\theta)$ is continuous and strictly positive on $\partial\mathbb{D}\setminus\{1\}$.  Theorem~\ref{T1.5} describes the asymptotic behavior at the critical point $\theta=0$ in terms of decay of $\alpha$, providing a sort of refinement of Weidmann's theorem. In a sense, Theorem~\ref{T1.5} is to Weidmann's theorem what Theorem~\ref{T1.1} is to Theorem~\ref{T1.2}.

Regarding our discussion around Conjecture~\ref{C1.3}, we should note that \cite{OPUC1} contains a more general conjecture than what we cited, with $(1-\cos\theta)^m$ replaced by an arbitrary nontrivial nonnegative trigonometric polynomial. In fact, \cite{SimonZlatos05,GolinskiiZlatos07} contain partial results towards the more general conjecture. However, in \cite{Lukic7}, we have disproved the general conjecture, proposing also a modification of the conjecture. The phenomenon that breaks the general conjecture depends on multiple zeros of the trigonometric polynomial, so it has no bearing to the cases in this paper.

In Section~\ref{S2}, we prove Theorem~\ref{T1.4}. In Section~\ref{S3}, we prove Theorem~\ref{T1.5}, using the method of \cite{Lukic1}.

Since the proof of Theorem~\ref{T1.1} only uses real-valued, power-law decaying test sequences \eqref{1.3}, a referee has pointed out that for such sequences, one can use Geronimus relations and results of Kreimer--Last--Simon~\cite{KreimerLastSimon09} to obtain exact asymptotics of $\log w(\theta)$ as $\theta \to 0$, thereby circumventing Theorem~\ref{T1.5}. We explain this approach in Section~\ref{S4}.

If, in the future, this test sequence approach is successfully implemented to higher-order Szeg\H o theorems with multiple critical points, the natural test sequences to use will be Wigner--von Neumann type potentials, of the type studied in \cite{Lukic1}. Theorem~\ref{T1.5} can easily be extended to the full class of potentials studied in \cite{Lukic1}, but only gives one-sided estimates of the asymptotic behavior at critical points. The alternative approach gives exact asymptotics for \eqref{1.3}, but the method of \cite{KreimerLastSimon09} has not been extended to more general Wigner--von Neumann type potentials. Therefore, at the current state of knowledge, we believe that it is of interest to have both approaches explained in this paper. We consider it an interesting open problem to extend the analysis of \cite{KreimerLastSimon09} to more general Wigner--von Neumann type potentials, which have one or more critical points inside the essential spectrum, and to understand the asymptotic behavior of the a.c.\ part of the measure near those critical points as well as near the endpoints $\pm 2$ of the essential spectrum. 

We are grateful to the referees for useful comments and, especially, for pointing out the alternative approach.

\section{Proof of Theorem~\ref{T1.4}} \label{S2}

Let us begin by noting that the trigonometric polynomial
\[
(1-\cos \theta)^m = \sum_{k=-m}^m b_k e^{-ik\theta}
\]
has $b_{-k} = b_k \in \mathbb{R}$ since it is even and real-valued, and
\[
b_0 = \int_0^{2\pi} (1-\cos \theta)^m \frac{d\theta}{2\pi} > 0.
\]
To investigate finiteness of the integral
\[
Z(\mu) = \int (1-\cos \theta)^m \log w(\theta) \frac{d\theta}{2\pi},
\]
we use a formula of Golinskii--Zlato\v s~\cite[Theorem 3.3]{GolinskiiZlatos07},
\begin{equation}\label{2.1}
Z(\mu) = \sum_{n=0}^\infty g(\alpha_{n-m}, \dots, \alpha_{n})
\end{equation}
where $g$ is given by
\begin{equation}\label{2.2}
g(\alpha_{n-m}, \dots, \alpha_{n}) = \Re \left( b_0 \log (1-\lvert\alpha_n\rvert^2) + \sum_{l=1}^m  \sum_{k=1}^l \sum_{t \in S_{l,k} } b_l N(t) \prod_{j=1}^k \alpha_{n-t_{2j-1}} \bar \alpha_{n-t_{2j}} \right).
\end{equation}
In this formula, we use the standard convention that $\alpha_{-1}=-1$ and $\alpha_n =0$ for $n\le -2$. 
$S_{l,k} \subset \{0,1,\dots, l\}^{2k}$ is a set of $2k$-tuples, and $N(t)$ are real-valued constants; \cite{GolinskiiZlatos07} gives descriptions of the sets and constants, which we do not repeat since we won't need such detailed information. All we need is that, apart from the logarithmic term, this is a finite sum of $k$-fold products with $k\le m$, which are multiplied by real constants.

We now set out to prove that we can replace each $k$-fold product in \eqref{2.2} by $\lvert \alpha_n \rvert^{2k}$ with a finite error on $Z(\mu)$. The motivation is simple: if $\delta \alpha$ is small, then the errors introduced by replacing $\alpha_{n+t_i}$ by $\alpha_n$ will also be small. Indeed, if we knew that $\delta \alpha \in \ell^1$, this would be trivial, since such replacements would introduce $\ell^1$ errors in \eqref{2.2}. Using only $\delta \alpha \in \ell^2$, this turns out to still be true, but the argument is more involved and uses the fact that \eqref{2.2} depends only on the real part of the product. It also uses a telescoping argument, so the error will be summable but not necessarily absolutely summable. 

\begin{lemma}\label{L2.1}
Let $t \in  \{0,1,\dots, l \}^{2k}$. Then for any $\alpha \in \mathbb{D}^\infty$ which obeys $\delta\alpha \in \ell^2$ and $\lim_{n\to\infty} \alpha_n =0$, the series 
\begin{equation}\label{2.3}
\sum_{n=0}^\infty \left( \Re \prod_{j=1}^k \alpha_{n-t_{2j-1}} \bar \alpha_{n-t_{2j}} -  \lvert \alpha_n \rvert^{2k}  \right)
\end{equation}
is convergent.
\end{lemma}

This lemma is the main technical step in our proof. For its proof, we need some preliminary inequalities.

\begin{lemma}\label{L2.2}
\begin{enumerate}[(i)]
\item For any $z_1, \dots, z_k, z'_1, \dots, z'_k \in \overline{\mathbb{D}}$,
\[
\lvert z_1 \cdots z_k - z'_1 \cdots z'_k \rvert \le k
\max_{j \in \{1, \dots, k\}} \lvert z_j - z'_j \rvert.
\]

\item For any $z_1, \dots, z_k \in \overline{\mathbb{D}}$,
\begin{equation}\label{2.4}
\left\lvert  \frac{z_1^k + \dots + z_k^k}k  -  z_1 \cdots z_k \right\rvert \le (k-1)^2 \max_{i, j \in \{1,\dots, k\}} \lvert z_i - z_j \rvert^2.
\end{equation}
\end{enumerate}
\end{lemma}

\begin{proof}
(i) This is immediate from the telescoping sum
\[
z_1 \cdots z_k - z'_1 \cdots z'_k = \sum_{j=1}^{k} (z_j - z'_j) \prod_{i=1}^{j-1} z_i  \prod_{l=j+1}^{k} z'_l.
\]

(ii) Denote
\[
P(z_1, \dots, z_k) = \frac{z_1^k + \dots + z_k^k}k -  z_1 \cdots z_k.
\]
We use (i) to estimate a partial derivative of $P$,
\begin{align*}
\left\lvert \frac{ \partial P}{\partial z_k} \right\rvert  & = \lvert z_k^{k-1} -  z_1 \cdots z_{k-1} \rvert  \le(k-1)  \max_{i,j \in \{1, \dots ,k\}} \lvert z_i - z_j \rvert.
\end{align*}
Note that the right-hand side doesn't increase if we replace $z_k$ by a value between its previous value and $z_1$. Thus, using the mean value theorem over an interval of length $\lvert z_k - z_1 \rvert$,
\[
\lvert P(z_1, \dots, z_{k-1} ,z_k) - P(z_1, \dots, z_{k-1}, z_1 ) \rvert \le (k-1)  \max_{i,j \in \{1, \dots ,k\}} \lvert z_i - z_j \rvert^2.
\]
Using this step $k-2$ more times to replace $z_{k-1}, \dots ,z_2$ by $z_1$, using the triangle inequality, and noting that $P(z_1, \dots ,z_1) = 0$, we conclude \eqref{2.4}.
\end{proof}

\begin{proof}[Proof of Lemma~\ref{L2.1}]
Since $t_i, t_j \in \{0, \dots, l \}$,
\begin{equation}\label{2.5}
\lvert \alpha_{n-t_{i}} - \alpha_{n-t_{j}} \rvert^2 \le \left( \sum_{q=0}^{l-1} \lvert \alpha_{n-q-1} - \alpha_{n-q} \rvert \right)^2 \le l  \sum_{q=0}^{l-1} \lvert \alpha_{n-q-1} - \alpha_{n-q} \rvert^2.
\end{equation}
Denote $\beta_n = \prod_{j=1}^k \alpha_{n-t_{2j-1}}$ and $\beta'_n = \prod_{j=1}^k \alpha_{n-t_{2j}}$. By Lemma~\ref{L2.2}(i) and \eqref{2.5},
\begin{equation}\label{2.6}
 \lvert\beta_n\rvert^2 +  \lvert \beta'_n \rvert^2 - 2 \Re \beta_n \bar\beta'_n = \lvert \beta_n - \beta'_n\rvert^2  \le (k-1)^2 l \sum_{q=0}^{l-1} \lvert \alpha_{n-q-1} - \alpha_{n-q} \rvert^2.
\end{equation}
By Lemma~\ref{L2.2}(ii) applied to $z_j = \lvert \alpha_{n-t_{2j}}\rvert^2$,
\[
\left \lvert \lvert \beta_n \rvert^2 -  \frac 1k \sum_{j=1}^k \lvert \alpha_{n-t_{2j}} \rvert^{2k} \right\rvert  \le (k-1)^2 \max_{i, j\in \{1,\dots,k\}} \left \lvert \lvert \alpha_{n-t_{2i}} \rvert^2 - \lvert \alpha_{n-t_{2j}} \rvert^2  \right\rvert^2.
\]
Since $\lvert\lvert z \rvert^2 - \lvert w \rvert^2 \rvert \le 2 \lvert z -w\rvert$ for $z,w \in \mathbb{D}$, using \eqref{2.5} we conclude
\begin{equation}\label{2.7}
\left \lvert \lvert \beta_n \rvert^2 -  \frac 1k \sum_{j=1}^k \lvert \alpha_{n-t_{2j}} \rvert^{2k} \right\rvert \le 4 (k-1)^2 l \sum_{q=0}^{l-1} \lvert \alpha_{n-q-1} - \alpha_{n-q} \rvert^2
\end{equation}
The analogous result holds for $\lvert\beta_n' \rvert^2$ by the same proof. Combining those with \eqref{2.6}, we conclude that
\begin{equation}\label{2.8}
\sum_{n=0}^\infty \left \lvert \Re  \prod_{j=1}^k \alpha_{n-t_{2j-1}} \bar \alpha_{n-t_{2j}} - \frac 1{2k} \sum_{i=1}^{2k} \lvert \alpha_{n-t_{i}} \rvert^{2k} \right\rvert \le \frac 92 (k-1)^2 l^2 \lVert \delta\alpha \rVert_2^2.
\end{equation}
It remains to replace each $t_i$ by $0$. It is in this step that we lose absolute convergence. Since $\alpha_n \to 0$, the following is a convergent telescoping series:
\[
 \sum_{n=0}^\infty ( \lvert \alpha_{n-t_{i}} \rvert^{2k} - \lvert \alpha_{n} \rvert^{2k} )  = \sum_{n=0}^{t_i - 1} \lvert \alpha_{n-t_i} \rvert^{2k} = 1.
\]
Combining this with \eqref{2.8} completes the proof.
\end{proof}

In the next lemma we apply this to the function $g$ of \eqref{2.2}.

\begin{lemma}\label{L2.3} There exists a power series in $\lvert \alpha_n \rvert^2$ convergent on $\mathbb{D}$,
\begin{equation}\label{2.9}
f(\alpha_n) = \sum_{k=1}^\infty c_k \lvert\alpha_n\rvert^{2k},
\end{equation}
such for any $\alpha \in \mathbb{D}^\infty$ which obeys $\delta\alpha \in \ell^2$ and $\lim_{n\to\infty} \alpha_n =0$, the series
\begin{equation}\label{2.10}
\sum_{n=0}^\infty \left( g(\alpha_{n-m}, \dots, \alpha_{n}) - f(\alpha_{n}) \right)
\end{equation}
is convergent. Moreover, $c_{m+1} < 0$.
\end{lemma}

\begin{proof}
Since $b_l N(t) \in \mathbb{R}$, we can apply Lemma~\ref{L2.1} to compare any $k$-fold product in \eqref{2.2} to $\lvert \alpha_n \rvert^{2k}$. We then collect same powers of $\lvert \alpha_n\rvert$ to compare the entire sum in \eqref{2.2} with
\[
\sum_{k=1}^m d_k \lvert \alpha_n \rvert^{2k}
\]
for some $d_1, \dots, d_m \in \mathbb{R}$. This proves that \eqref{2.10} is convergent with
\[
f(\alpha_n) =  b_0  \log (1 - \lvert \alpha_n \rvert^2)  + \sum_{k=1}^m d_k \lvert \alpha_n \rvert^{2k}.
\]
The proof is completed by using $b_0 > 0$ and
\[
 \log (1 - \lvert \alpha_n \rvert^2) = - \sum_{k=1}^\infty \frac{1}{k} \lvert \alpha_n \rvert^{2k}. \qedhere
\]
\end{proof}

\begin{proof}[Proof of Theorem~\ref{T1.4}]
If \eqref{1.1} holds, then $w>0$ for Lebesgue-a.e.\ $\theta$ so, by Rakhmanov's theorem \cite{Rakhmanov83},
\begin{equation}\label{2.11}
\lim_{n\to\infty} \alpha_n = 0.
\end{equation}
If $\alpha \in \ell^{2l}$, \eqref{2.11} holds trivially. Thus, both sides of the equivalence imply \eqref{2.11}, so we may work under that assumption.

Let $l$ be given by the condition that $c_l$ is the first non-zero coefficient in the series \eqref{2.9}. Since $c_{m+1} \neq 0$, it follows that $l\le m+1$.

We first prove by contradiction that $c_l<0$. Assume $c_l>0$. Since
\begin{equation}\label{2.12}
\lim_{n\to \infty} \frac{ f(\alpha_n) }{ c_l \lvert \alpha_n\rvert^{2l} } = 1,
\end{equation}
for any sequence $\alpha \notin \ell^{2l}$ with $\delta \alpha \in \ell^2$ and \eqref{2.11}, \eqref{2.10} would imply $Z(\mu) = + \infty$, which contradicts Remark~\ref{R1.1}.

Thus, $c_l < 0$. Using \eqref{2.12} and applying the limit comparison test again, together with \eqref{2.10}, implies that $Z(\mu) > -\infty$ is equivalent to  $\alpha \in \ell^{2l}$.
\end{proof}

\section{Proof of Theorem~\ref{T1.5}} \label{S3}

We will explain how Theorem~\ref{T1.5} follows from estimates implicit in \cite{Lukic1}, summarizing the method but referring to that paper for details.

The proof uses Pr\" ufer variables for the unit circle, which are defined for $z=e^{i\eta}$ with $\eta\in\mathbb{R}$ by $r_n(\eta) > 0$, $\theta_n(\eta) \in\mathbb{R}$, and
\begin{equation}\label{3.1}
\varphi_n(e^{i\eta}) = r_n(\eta) e^{i[n\eta+\theta_n(\eta)]}.
\end{equation}
A special case of the result in \cite{Lukic1} is that, under the assumptions of Theorem~\ref{T1.5}, $\log r_n(\eta)$ converges as $n\to \infty$ uniformly on intervals away from $\eta=0$. Thus, by Bernstein--Szeg\H o approximations,
\begin{equation}\label{3.2}
\log w(\eta) = - 2 \lim_{n\to\infty} \log r_n(\eta), \quad \eta \in (0,2\pi).
\end{equation}
To characterize the asymptotic behavior of $w$ as $\eta \to 0$, we must take a closer look at the method in \cite{Lukic1}. The estimates used there to prove existence of the limit \eqref{3.2}  also give  a bound on the limit.

The Szeg\H o recursion relation implies a first order recurrence relation for the Pr\"ufer variables,
\begin{equation}\label{3.3}
\log \frac{r_{n+1}}{r_n} + i  (\theta_{n+1}-\theta_n)  =  \log(1 - \bar\alpha_n e^{-i[(n+1)\eta + 2 \theta_n]}) - \tfrac 12 \log(1-\lvert \alpha_n \rvert^2).
\end{equation}
Taking the real part, summing in $n$, and using $r_0=1$, we investigate convergence of
\begin{equation}\label{3.4}
\log r_N = \Re \sum_{n=0}^{N-1} \left( \log(1 - \bar\alpha_n e^{-i[(n+1)\eta + 2 \theta_n]}) - \tfrac 12 \log(1-\lvert \alpha_n \rvert^2) \right)
\end{equation}
as $N\to\infty$. By expanding the logs in \eqref{3.4} and ignoring terms of order $O(\lvert \alpha_n \rvert^p)$ (which are summable uniformly in $\eta$ since $\alpha \in \ell^p$), the goal becomes to control a finite linear combination of sums of the form
\begin{equation}\label{3.5}
\sum_{n=0}^{N-1} \alpha_n^I \bar\alpha_n^J e^{i(I-J)[(n+1)\eta + 2 \theta_n]}
\end{equation}
with $I-1 \ge J \ge 0$ and $I+J < p$. We control such terms by a more quantitative version of \cite[Lemma 6.1]{Lukic1}.

\begin{lemma}\label{L3.1}
Let $k\in\mathbb{Z}$ and $\phi\in [0,2\pi)$, with $k$ and $\phi$ not both equal to $0$. If $\{e^{i\phi n} \Gamma_n\}_{n=0}^\infty$ has bounded variation and $\Gamma_n \to 0$, and functions $f(\eta)$, $g(\eta)$ are such that
\[
g(\eta) = \frac{f(\eta)}{e^{-i(k\eta -\phi)}-1},
\]
then the series
\[
S =  \sum_{n=0}^{\infty} \left( f(\eta) \Gamma_n e^{ik[(n+1)\eta+2\theta_n]}  - g(\eta)  \Gamma_n e^{ik[(n+1)\eta+2\theta_n]} \left( e^{2ik(\theta_{n+1}-\theta_n)}-1\right) \right) 
\]
is convergent and
\[
\lvert S \rvert \le 2 \lvert g(\eta)\rvert  \sum_{n=0}^{\infty} \lvert e^{i\phi}  \Gamma_{n+1}  - \Gamma_{n} \rvert.
\]
\end{lemma}

\begin{proof}
If we take the partial sum $\sum_{n=M}^{N-1}$, algebraic manipulations bring it to the form
\[
g(\eta) \left( \sum_{n=M}^{N-1} e^{i\phi} \Gamma_n  e^{ik[n\eta+2\theta_n]} -  \sum_{n=M}^{N-1}  \Gamma_n e^{ik[(n+1)\eta+2\theta_{n+1}]} \right).
\]
Changing the index $n$ to $n+1$ in the first sum and recombining them, we see that this is equal to
\[
g(\eta) \left( e^{i\phi} \Gamma_M e^{ik[M\eta + 2\theta_M]} - \Gamma_{N-1} e^{ik[N\eta+2\theta_N]}  +  \sum_{n=M}^{N-2} \left( e^{i\phi} \Gamma_{n+1} -  \Gamma_n \right) e^{ik[(n+1)\eta+2\theta_{n+1}]} \right)
\]
so it is bounded in absolute value by
\[
\lvert g(\eta) \rvert \left( \lvert \Gamma_M \rvert + \lvert \Gamma_{N-1} \rvert + \sum_{n=M}^{N-2} \lvert e^{i\phi} \Gamma_{n+1} - \Gamma_n \rvert \right).
\]
The claim now follows easily if we use
\[
\lvert \Gamma_L \rvert = \left\lvert \sum_{n=L}^\infty (e^{i\phi (n+1)} \Gamma_{n+1} - e^{i\phi n} \Gamma_n ) \right\rvert \le \sum_{n=L}^\infty \lvert e^{i\phi}\Gamma_{n+1} - \Gamma_n \rvert. \qedhere
\]
\end{proof}

In our method, the term \eqref{3.5} appears multiplied by a factor denoted $f_{I,J,I-J,0}(\eta;0,\dots,0;0,\dots,0)$ in the notation of Sections 8--9 of \cite{Lukic1}. By Lemma~\ref{L3.1}, the difference between it and
\begin{equation}\label{3.6}
g_{I,J,I-J,0}(\eta;0,\dots,0;0,\dots,0) \sum_{n=0}^{N-1} \alpha_n^I \bar\alpha_n^J e^{i(I-J)[(n+1)\eta + 2 \theta_n]} \left( e^{2i(I-J)(\theta_{n+1} - \theta_n)} - 1 \right)
\end{equation}
is bounded uniformly in $N$ and in $e^{i\eta} \in \partial\mathbb{D}$. The factor $e^{2i(I-J)(\theta_{n+1} - \theta_n)} - 1$ is of order $O(\lvert \alpha_n\rvert)$ and can also be written down as a sum of terms of the form \eqref{3.5}, which is used to drive an iterative scheme: replacing \eqref{3.5} by \eqref{3.6}, using \eqref{3.3} to write \eqref{3.6} as a finite linear combination of sums of the form \eqref{3.5} with strictly greater $I+J$, and iterating; see \cite{Lukic1} for details. From the recurrence relations for $f$'s and $g$'s ((8.10) and (8.13) of \cite{Lukic1}) it follows by induction that as $\eta \to 0$,
\begin{align*}
f_{I,J,I-J,0}(\eta;0,\dots,0; 0 \dots,0) & = O(\lvert \eta\rvert^{-(I+J-1)}), \\
g_{I,J,I-J,0}(\eta; 0,\dots, 0; 0 \dots,0) & = O(\lvert \eta \rvert^{-(I+J)}).
\end{align*}
Our estimate of $\log r_N$ uses $g$'s with $I+J \le p-1$, which are all $O(\lvert \eta \rvert^{1-p})$. In the end, there are remaining terms containing $p$-fold products of $\alpha$'s, which are summable by $\alpha \in \ell^p$; they are preceded by $f$'s with $I+J \le p$, which are also all $O(\lvert \eta \rvert^{1-p})$. Thus, the proof is complete.

\section{Asymptotic behavior of the a.c.\ part of the measure for real-valued power-law decaying Verblunsky coefficients} \label{S4}

We begin by reviewing the basic properties of two useful transformations of measures on the unit circle: sieving and the Szeg\H o mapping. See \cite[Section 1.6]{OPUC1} and \cite[Section 13.1]{OPUC2}, respectively, for details.

If $\mu$ is a probability measure on $\partial\mathbb{D}$ and $N$ a positive integer, the sieved measure $\mu^{\{N\}}$ is given by
\[
d\mu^{\{N\}} (\theta) = \frac 1N d\mu(N\theta).
\]
Its Verblunsky coefficients, $\alpha^{\{N\}}_n$, are given in terms of the original Verblunsky coefficients as
\[
\alpha^{\{N\}}_n = \begin{cases} \alpha_{(n-N+1)/N} & n\equiv N-1 \pmod{N} \\
0 &  n\not\equiv N-1 \pmod{N}
\end{cases}
\]

We now describe the Szeg\H o mapping. Let $\gamma$ be a probability measure supported on $[-2,2]$. The Szeg\H o mapping relates $\gamma$ to a probability measure $\nu$ supported on $\partial \mathbb{D}$ which is symmetric with respect to complex conjugation ($\theta \mapsto -\theta$) and such that for Borel functions $g: [-2,2] \to \mathbb{R}$,
\[
\int_{[0,2\pi]} g(2 \cos \theta) d\nu(\theta) = \int_{[-2,2]} g(x) d\gamma(x).
\]
If $d\gamma(x) = f(x) dx + d\gamma_\s$ and $d\nu(\theta) = v(\theta) \frac{d\theta}{2\pi} + d\nu_\s$ are Lebesgue decompositions of $\gamma$ and $\nu$, then
\begin{equation}\label{4.1}
v(\theta) = 2 \pi \lvert \sin\theta\rvert f(2\cos \theta).
\end{equation}
Geronimus discovered that the Jacobi parameters $\{a_n, b_n\}_{n=1}^\infty$ of $\gamma$ can be expressed in terms of Verblunsky coefficients $\{\beta_n\}_{n=0}^\infty$ of $\nu$, as
\begin{align*}
a_{n+1}^2 & = (1-\beta_{2n-1})(1-\beta_{2n}^2)(1+\beta_{2n+1}) \\
b_{n+1} & = (1-\beta_{2n-1}) \beta_{2n} - (1+\beta_{2n-1}) \beta_{2n-2}
\end{align*}
with the convention $\beta_{-1} = -1$.

These two constructions combine nicely. Let us start with the measure $\mu$ with real-valued Verblunsky coefficients $\{\alpha_n\}_{n=0}^\infty$, take $\nu = \mu^{\{2\}}$ and the measure $\gamma$ on $[-2,2]$ which corresponds to $\nu$ via the Szeg\H o mapping. It has
\begin{align*}
a_{n+1}^2 & = (1 - \alpha_{n-1})(1+\alpha_n)
\end{align*}
and $b_n = 0$.

In particular, if
\[
\alpha_n = (n+2)^{-\tau}
\]
for some $\tau \in (0,1)$, it is easy to see that $a_n \le a_{n+1} \le 1$ and $\lim_{n\to\infty} a_n = 1$, so the measure $\gamma$ is exactly of the form considered in Section 3 of Kreimer--Last--Simon~\cite{KreimerLastSimon09}. They introduce for $x\in (-2,2)$ the integer
\[
N(x) = \sup \{ n\in \mathbb{N} \mid 2 a_n \le \lvert x \rvert \}
\]
as well as
\[
\gamma_n(x) = \arccosh \frac{\lvert x\rvert}{2 a_n}, \qquad \text{for }n \le N(x),
\]
and they prove that
\[
\left\lvert - \frac 12 \log f(x)  - \sum_{n=1}^{N(x)} \gamma_n(x) \right\rvert \le C + \log N(x) - \log(a_{N(x)+2} - a_{N(x)+1})
\]
for some $x$-independent constant $C$.

Since $a_n$ behaves asymptotically as
\[
2 - 2 a_n  = n^{-2 \tau} + O(n^{-\tau-1}), \qquad n\to \infty,
\]
and
\[
a_{n+2} - a_{n+1} \sim \tau (n+2)^{-1 - 2 \tau},
\]
we conclude that
\[
N(x) \sim (2 - \lvert x\rvert)^{1/(2\tau)}, \qquad x \to \pm 2
\]
and therefore
\[
\left\lvert - \frac 12 \log f(x)  - \sum_{n=1}^{N(x)} \gamma_n(x) \right\rvert = O(\log(2 - \lvert x\rvert)), \qquad  x  \to \pm 2.
\]
Denoting $\delta = 2 - \lvert x \rvert$ and $\delta_n = 2 - 2 a_n - \delta$, we can write
\[
\sum_{n=1}^{N(x)} \gamma_n(x) = \sum_{n=1}^{N(x)} F\left(\frac{\delta_n}{2-\delta}\right) 
\]
where $F(z) = \arccosh( 1 / (1-z))$. Using the Taylor expansion of $F(z)$ similarly to the Example 5.3 of \cite{KreimerLastSimon09}, we can conclude that
\[
\log f(x) \sim -2 \sum_{n=1}^{N(x)} \gamma_n(x)  \sim C (2 - \lvert x \rvert)^{1/2 - 1/(2\tau)}, \qquad x \to \pm 2.
\]
Using \eqref{4.1}, this implies that for the original measure on $\partial\mathbb{D}$,
\[
\log w(\theta) \sim C \lvert \sin \theta \rvert^{1 - 1/\tau}, \qquad \theta \to 0.
\]
Thus, $ \sin^{2m} \theta \log w(\theta)$ is integrable if and only if $1/\tau < 2m+2$. Since $\alpha \in \ell^{2m+2}$ if and only if $1/\tau < 2m+2$, this provides an alternative proof of Theorem~\ref{T1.1} given Theorem~\ref{T1.4}.

\bibliographystyle{amsplain}

\providecommand{\bysame}{\leavevmode\hbox to3em{\hrulefill}\thinspace}
\providecommand{\MR}{\relax\ifhmode\unskip\space\fi MR }
\providecommand{\MRhref}[2]{%
  \href{http://www.ams.org/mathscinet-getitem?mr=#1}{#2}
}
\providecommand{\href}[2]{#2}

\end{document}